\documentclass[noamsfonts]{amsart}
\usepackage{setspace}
\usepackage[charter,expert]{mathdesign}
\usepackage[osf]{XCharter}
\selectfont
\usepackage[activate={true,nocompatibility},final]{microtype}

\usepackage[all, arc]{xy}
\SelectTips{eu}{}
\entrymodifiers={+!!<0pt,\fontdimen22\textfont2>}

\usepackage[pdftex, colorlinks=true, linkcolor=blue, citecolor=magenta, linktocpage]{hyperref}

\theoremstyle{plain}
\newtheorem{thm}[subsubsection]{Theorem}

\newtheorem{prop}[subsubsection]{Proposition}

\theoremstyle{definition}

\newtheorem{remark}[subsubsection]{Remark}

\theoremstyle{definition}

\numberwithin{equation}{subsubsection}

\def\11{\mathbf{1}}

\def\AA{\mathbf{A}} 
 
\def\CC{\mathbf{C}}

\def\QQ{\mathbf{Q}}

\def\coker{\mathrm{coker}}

\def\DM{\mathcal{D}}

\def\Gr{\mathrm{Gr}}

\def\MHM{\mathcal{M}}

\def\pH{\vphantom{H}^p\!H}

\def\Spec{\mathrm{Spec}}

\def\pupsi{\vphantom{H}^p\Psi^{\mathrm{u}}}
\def\red{\mathrm{red}}

\newcommand{\mapright}[1]{\xrightarrow{#1}}

\makeatletter
\newcommand{\holim@}[2]{%
      \vtop{\m@th\ialign{##\cr
                  \hfil$#1\operator@font holim\,$\hfil\cr
                      \noalign{\nointerlineskip\kern1.5\ex@}#2\cr
                          \noalign{\nointerlineskip\kern-\ex@}\cr}}%
                      }
                      \newcommand{\holim}{%
                            \mathop{\mathpalette\holim@{\leftarrowfill@\textstyle}}\nmlimits@
                        }
                        \makeatother

                        \title{Hodge-Grothendieck classes and monodromy invariants of nearby cycles sheaves}
\author{R. Virk}
\begin{document}
\maketitle
\renewcommand{\thesubsection}{\textbf{\arabic{subsection}}}
\renewcommand{\thesubsubsection}{\textbf{\arabic{subsection}.\arabic{subsubsection}}}

This note was prompted by a question of G. Williamson \cite{W} (paraphrasing): \emph{``How much information about nearby cycles can we deduce without knowing the defining function?''}. The idea being that the associated graded of the monodromy filtration\footnote{In the sense of \cite[1.6]{D}.} on (unipotent) nearby cycles is determined by its primitive part. The latter should be completely determined by the central fibre.

We pursue this in the Hodge theoretic context.
Theorem \ref{hdgclass} shows that the Hodge-Grothendieck class of (the unipotent part of) nearby cycles sheaves is independent of the defining equation. 
Theorem \ref{locinv} is a generalized local invariant cycles result (a variant of \cite[Th\'eor\`eme 3.6.1]{D}).
No claims to orginality are being made.

\subsection{Notation}
\subsubsection{}We write $\MHM(X)$ for the category of mixed Hodge modules on a variety\footnote{`Variety' = `separated scheme of finite type over $\mathrm{Spec}(\CC)$'.} $X$, and $\DM(X)$ for its bounded derived category. The cohomology functors associated to the evident t-structure (with heart $\MHM(X)$) are denoted $\pH^i\colon \DM(X) \to \MHM(X)$. The $d$-th Tate twist will be denoted by $(d)$.
Functors on derived categories will always be derived. I.e., we write $f_*$ instead of $Rf_*$, etc.

\subsubsection{}Let $f\colon X \to \AA^1$ be a morphism of varieties. Set $X_0 = f^{-1}(0)$ and $X^* = X - X_0$. We write $i\colon X_0 \to X$ and $j\colon X^*\to X$ for the inclusions.
\[ \xymatrix{X_0\ar[d] \ar[r]^i & X\ar[d]^f & X^*\ar[l]_j\ar[d]\\
\{0\}\ar[r] & \AA^1 & \AA^1 - \{0\}\ar[l]}\]
The unipotent part of the nearby cycles functor associated to $f$ is denoted:
\[ \psi_f^u\colon \DM(X)\to \DM(X_0) \]
Shift convention is that $\psi_f^u[-1]$ is t-exact.
We write:
\[ N \colon \psi_f^u \to \psi_f^u(-1) \]
for the log of the unipotent part of monodromy.

\subsection{Preliminaries}
\subsubsection{}For $M\in \DM(X^*)$ it is convenient to set:
\[\pupsi_f(M) = \psi_f^u(j_*M)[-1]. \]
In particular, $\pupsi_f$ restricts to a t-exact functor $\MHM(X^*) \to \MHM(X_0)$.\footnote{Our $\pupsi_f$ is denoted $\psi_{f,1}$ in \cite{S89}.}

\subsubsection{}
We have a canonical distinguished triangle (see \cite[Remark 5.2.2]{S88}):
\begin{equation}\label{dt} \pupsi_f(M) \mapright{N} \pupsi_f(M)(-1) \to i^*j_*M \mapright{[1]} \end{equation}
For $M\in \MHM(X^*)$ this yields an exact sequence:
\begin{equation}\label{exactseq} 0 \to \pH^{-1}(i^*j_*M) \to \pupsi_f(M) \mapright{N} \pupsi_f(M)(-1) \to \pH^0(i^*j_*M) \to 0 \end{equation}

\begin{prop}\label{kerprop}Let $M\in \MHM(X^*)$. Let $j_{!*}(M) \in \MHM(X)$ be the intermediate extension of $M$ to $X$. Then we have canonical isomorphisms:
\begin{enumerate}
\item $ \ker(N) \simeq i^*j_{!*}(M)[-1]$
\item $ \coker(N) \simeq i^!j_{!*}(M)[1]$
\end{enumerate}
\end{prop}

\begin{proof}We only show (i). The proof of (ii) is similar. Apply $i^*$ to the canonical distinguished triangle:
\[ i_*i^!(j_{!*}(M)) \to j_{!*}(M) \to j_*j^*(j_{!*}(M)) \mapright{[1]} \]
to get the distinguished triangle:
\[ i^!j_{!*}(M) \to i^*j_{!*}(M) \to i^*j_*M \mapright{[1]} \]
This yields the exact sequence:
\[ \pH^{-1}(i^!j_{!*}(M)) \to \pH^{-1}(i^*j_{!*}(M)) \to \pH^{-1}(i^*j_*M) \to \pH^0(i^!j_{!*}(M)). \]
As $i^!$ is left t-exact, the left most term must vanish. Additionally, since  $j_{!*}(M)$ is the intermediate extension of $M$, the right most term must also vanish. In other words, $\ker(N)\simeq \pH^{-1}(i^*j_{!*}(M))$. On the other hand, $\pH^k(i^*j_{!*}(M)) = 0$ for $k\neq -1$.
\end{proof}

\subsection{Weights}
\subsubsection{}Mixed Hodge modules come equipped with a finite increasing filtration (the weight filtration) which we denote by $W_{\bullet}$. The associated graded is denoted $\Gr^W_{\bullet}$. Morphisms in $\MHM(X)$ are \emph{strictly} compatible with $W_{\bullet}$ \cite[1.5]{S89}.

\subsubsection{}An object $M\in \DM(X)$ is said to have weights $\leq n$ (resp. $\geq n$) if $\Gr^W_k\pH^i(M) = 0$ for $k >i+n$ (resp. $k< i+n)$. The object $M$ is called pure of weight $n$ if $\Gr^W_k \pH^i(M) = 0$ for $k\neq i+n$.

\subsubsection{}Let $M\in \MHM(X^*)$. Then:
\[ NW_k\pupsi_f(M) \subset W_{k-2}\pupsi_f(M)(-1), \]
 for each $k$. Additionally, if $M$ is pure of weight $n$, then:
\[ N^k \colon \Gr_{n-1+k}^W\pupsi_f(M) \mapright{\sim} \Gr_{n-1-k}^W\pupsi_f(M)(-k) \]
is an isomorphism for each $k\geq 0$ (see \cite[1.19]{S89} and \cite[1.6]{D}). Further, we have an isomorphism:
\[ \Gr^W_k\ker(N) \simeq \ker(N\colon \Gr^W_k\pupsi_f(M) \to \Gr^W_{k-2}\pupsi_f(M)(-1)). \]
Consequently, for $M$ pure of weight $n$, we have an isomorphism:
\begin{equation}\label{primeq} \Gr^W_k\pupsi_f(M) \simeq \bigoplus_{\stackrel{m \geq |n-1-k|,}{m = n-1-k \pmod 2}}\Gr_{n-1-m}^W\ker(N)((n-1-m-k)/2)\end{equation}

\subsection{Hodge-Grothendieck classes}
\subsubsection{}Let $K_0(X_0)$ denote the Grothendieck group of $\DM(X_0)$ (equivalently $\MHM(X_0)$).
\begin{thm}\label{hdgclass}
Let $f,g\colon X\to \AA^1$ be morphisms of varieties. If $f^{-1}(0)_{\red} = g^{-1}(0)_{\red}$, then $\psi_f^u$ and $\psi_g^u$ define the same map on Grothendieck groups. I.e., in $K_0(X_0)$:
\[ [\psi_f^u(M)] = [\psi_g^u(M)], \]
for each $M\in \DM(X)$
\end{thm}

\begin{proof}
It suffices to show $[\pupsi_f(M)] = [\pupsi_g(M)]$ for $M\in \MHM(X^*)$.
We may also assume that $M$ is pure of weight $n$. By Proposition \ref{kerprop}(i) and \eqref{primeq}:
\begin{align*}
[\pupsi_f(M)] &= \sum_k \sum_{\stackrel{m \geq |n-1-k|,}{m = n-1-k \pmod 2}}[\Gr_{n-1-m}^Wi^*j_{!*}M[-1]((n-1-m-k)/2)] \\ 
&= [\pupsi_g(M)] \qedhere \end{align*}
\end{proof}

\begin{remark}If only the ordinary Grothendieck class of $\psi_f^u$ is of interest (i.e., working with constructible sheaves as opposed to mixed Hodge modules), then Hodge theory may be completely avoided as follows. Replace $W_{\bullet}$ by the monodromy filtration associated to $N$ and argue exactly as above. The identity \eqref{primeq} holds. The only extra ingredient needed in the proof of Theorem \ref{hdgclass} is that the induced filtration on $i^*j_{!*}(M)$ is independent of $f$.\footnote{This is the heart of the matter in the proof of Theorem \ref{hdgclass}. It is somewhat obscured by the expository choice of not making it totally explicit that the monodromy filtration and weight filtration coincide.} A purely topological demonstration of this is the main result of \cite{ELM}.\footnote{The point is that the monodromy filtration induced from nearby cycles coincides with that induced from Verdier specialization. Verdier specialization does not depend on $f$.}
\end{remark}

\pagebreak
\subsection{Invariant cycles}
\subsubsection{}The adjunction map $M\to j_*j^*M$ along with \eqref{dt} yields canonical maps:
\[ i^*M \to i^*j_*j^*M \to \psi_f^u(M).\]
\begin{thm}[Local invariant cycles]\label{locinv}
Let $M\in \DM(X)$ be pure. If $f\colon X\to \AA^1$ is proper, then the sequence of maps:
\[ H^k(X_0; i^*M) \to H^k(X_0; \psi_f^u(M)) \mapright{N} H^k(X_0; \psi_f^u(M))(-1) \]
is exact for each $k$.
\end{thm}

\begin{proof}
Say $M$ is pure of weight $n$. As pushforward along a proper map preserves weights \cite[1.8]{S89}, the weight filtration on $H^k(X_0; \psi^u_f(M))$ is the monodromy filtration\footnote{In the sense of \cite[1.6]{D}} centered at $n+k$ \cite[(5.3.4.2)]{S88}, \cite[1.19]{S89}. In particular, $\ker(N)$ has weights $\leq n+k$.
From \eqref{dt} we infer that $\ker(N)$ is the image of:
\[ H^{k}(X_0; i^*j_*j^*M) \to H^k(X_0; \psi_f^u(M)). \]
 Thus, it suffices to show $H^k(X_0; i^*M) \to H^k(X_0; i^*j_*j^*M)$ is surjective on weights $\leq n+k$. The map $i^*M\to i^*j_*j^*M$ fits into a distinguished triangle:
\[ i^!M \to i^*M \to i^*j_*j^*M \mapright{[1]} \]
This yields an exact sequence:
\[ H^k(X_0; i^*M) \to H^k(X_0; i^*j_*j^*M) \to H^{k+1}(X_0; i^!M). \]
As $i^!$ does not lower weights \cite[1.7]{S89}, the right most term has weights $\geq n+k+1$. Hence, $H^k(X_0; i^*M) \to H^k(X_0; i^*j_*j^*M)$ must be surjective on weights $\leq n+k$.
\end{proof}
\begin{remark}
If $X$ is rationally smooth, then taking $M$ to be the constant sheaf on $X$ (with trivial Hodge structure)\footnote{More precisely, $\MHM(\Spec(\CC))$ is the category of polarizable $\QQ$-mixed Hodge structures \cite[1.4]{S89}. Let $\QQ^H$ be the trivial 1-dimensional, weight $0$ Hodge structure. Let $a\colon X\to \Spec(\CC)$ be the structure map. Take $M=a^*\QQ^H$.} recovers the classical local invariant cycles theorem.
\end{remark}
\begin{remark}Using an argument similar to the proof above, the exact sequence of Theorem \ref{locinv} may be extended to a generalized Clemens-Schmid exact sequence.
\end{remark}

\end{document}